\documentclass{amsart}

\usepackage[utf8]{inputenc}

\usepackage{epstopdf}% To incorporate .eps illustrations using PDFLaTeX, etc.
\usepackage[caption=false]{subfig}% Support for small, `sub' figures and tables

\theoremstyle{plain}% Theorem-like structures provided by amsthm.sty
\newtheorem{theorem}{Theorem}[section]

\theoremstyle{definition}

\theoremstyle{remark}

\numberwithin{equation}{section}

\usepackage{scalerel}
\usepackage{amsfonts,amssymb,amsmath,empheq}
\usepackage{amsthm}
\usepackage{booktabs}
\usepackage{graphics}
\usepackage{graphicx,color}
\usepackage{latexsym}
\usepackage{subfig}
\usepackage{multirow}
\usepackage{float}
\usepackage{cases}
\usepackage{bm}
\usepackage{color}
\usepackage{indentfirst}
\usepackage{epstopdf}
\usepackage{CJK}
\usepackage[colorlinks,
            linkcolor=red,
            anchorcolor=blue,
            citecolor=green
            ]{hyperref}
            
\newcommand{\beqa}{\begin{eqnarray}}
\newcommand{\eeqa}{\end{eqnarray}}
\newcommand{\beqas}{\begin{eqnarray*}}
\newcommand{\eeqas}{\end{eqnarray*}}
\newcommand{\pr}{\partial}
\newcommand{\prg}{\nabla}
\newcommand{\bx}{\mathbf{x}}
\newcommand{\bd}{\mathbf{d}}

\newcommand{\beq}{\begin{equation}}
\newcommand{\eeq}{\end{equation}}
\newcommand{\beqs}{\begin{equation*}}
\newcommand{\eeqs}{\end{equation*}}
\newcommand{\B}{\mathbf}
\newcommand{\bW}{\B{W}}
\newcommand{\bM}{\B{M}}
\newcommand{\sbM}{{\scaleto{\bM}{4pt}}}
\newcommand{\sbW}{{\scaleto{\bW}{4pt}}}

\newcommand{\blam}{\boldsymbol{\lambda}}
\newcommand{\bmu}{\boldsymbol{\mu}}

\newcommand{\la}{\langle}
\newcommand{\ra}{\rangle}
\def\R{{\mathbb{R}}}

\def\DIV{\mathrm{div}}
\def\DIAG{\operatorname{diag}}
\usepackage{xcolor}

\colorlet{colorYO}{red}
\colorlet{colorAL}{blue}
\begin{document}

%
% HEADERS
%
\title[Time-domain FWI for poroelastic equations]{On the time-domain full waveform inversion for time-dissipative and dispersive poroelastic media}
% %
 \author{Miao-jung Yvonne Ou}
 \address{Department of Mathematical Sciences, University of Delaware,
 508 Ewing Hall, Newark, Delaware 19716 USA}
 \email{mou@udel.edu}
% %
 \author{Petr Plech\'a\v{c}}
 \address{Department of Mathematical Sciences, University of Delaware,
510 Ewing Hall, Newark, Delaware 19716 USA}
\email{plechac@udel.edu}
% %
\author{Jiangming Xie}
\address{Department of Mathematical Sciences, Tsinghua University, Beijing 10084, China}
 \email{xiejiangming@mail.tsinghua.edu.cn}
\date{\today}
\maketitle
\begin{center}
\vspace{-0.2in}

\emph{Dedicated to Professor Robert P. Gilbert on the occasion of his 90th birthday}
 \end{center}

%
\iffalse
% TAYLOR STYLE FOR HEADER
%
\title{On the time-domain full waveform inversion for time-dissipative and dispersive poroelastic media}
\author{
\name{Miao-jung Yvonne Ou\textsuperscript{a}\thanks{corresponding author, email:mou@udel.edu}, Petr Plech\'a\v{c}\textsuperscript{a} and Jiangming Xie\textsuperscript{b}} 
\affil{\textsuperscript{a}Department of Mathematical Sciences, University of Delaware, Newark, DE 19716, USA;\newline
\textsuperscript{b}Department of Mathematical Sciences, Tsinghua University, Beijing 10084, China}
}

\maketitle

\begin{flushleft}
\emph{Dedicated to Professor Robert P. Gilbert on the occasion of his 90th birthday}
\end{flushleft}
\fi
%%%%%%
\vspace{-0.2in}

\begin{abstract}
This paper concerns the Time-Domain Full Waveform Inversion (FWI) for dispersive and dissipative poroelastic materials. The forward problem is an initial boundary value problem (IBVP) of the poroelastic equations with a memory term; the FWI is formulated as a minimization problem of a least-square misfit function with the (IBVP) as the constraint. In this paper, we derive the adjoint problem of this minimization problem, whose solution can be applied to computed the direction of steepest descent in the iterative process for minimization. The adjoint problem has a similar numerical structure as the forward problem and hence can be solved by the same numerical solver. Because the tracking of the energy evolution plays an important role in the FWI for dissipative and dispersive equations, the energy analysis of the forward system is also carried out in this paper.
\end{abstract}

\vspace{0.1in}

%35K90, % Abstract parabolic equations 
%35M13, % Initial-boundary value problems for PDEs of mixed type
%35Q74, % PDEs in connection with mechanics of deformable solids
%35R60, % PDEs with randomness, stochastic partial differential equations
%74E35, % Random structure in solid mechanics

%
% SECTION
% 
%
\section{Introduction}
For the application of poroelastic wave equations in biomedical and geological research, a time-domain simulation is more feasible because the input wave signal is usually a pulse of finite band-width, instead of a monochromatiic signal. The dispersive nature of a poroelastic material is encoded in the memory term of the poroelastic equations.  The presence of the memory term poses a significant challenge in developing a time domain numerical solver for the poroelastic equations. This challenge can be handled by the method of auxiliary variables as long as a high precision approximation of the memory kernel function can be obtained. Initiated in \cite{ou2014on-reconstructi} and further developed In \cite{ou2018augmented}, the link between the Laplace transform of the memory kernel and the Herglotz-Nevanlinna functions was utilized to obtain a highly accurate pole-residue approximation, which makes it possible for replacing the memory terms with sums of auxiliary variables that satisfy ordinary differential equations. The new system is termed \emph{augmented Biot equations} because it contains the new auxiliary variables and the corresponding  governing equations. By design, the augmented system has no explicit memory terms and can be solved much more efficiently than the original system.

Incorporating the dissipative and the dispersive mechanism of the fluid saturated poroelastic media in the forward wave solver is important in recovering poroelastic material properties from the probing waves (the inverse problem).  The full waveform inversion (FWI) has become a powerful tool in the geological exploration since the 1980s \cite{doi:10.1190/geo2013-0030.1,Tarantola1982Generalized-Non,Jeroen-Tromp2005Seismic-tomogra,Pengliang-Yang2016A-review-on-the,Qingjie-Yang2021Frequency-domai,Engquist2019Seismic-inversi,Engquist2019Seismic-imaging}. FWI has been successfully applied in seismic exploration with  (visco)-elastic  models \cite{Pengliang-Yang2016A-review-on-the} in recovering certain material properties. However, certain important poroelastic properties such as porosity and relative velocity between fluid and solid cannot be recovered from these models. For example, it is reported  in \cite{Chotiros2002An-inversion-fo} that the measured parameters of water-saturated sand in the laboratory is better predictions by a poroelastic Biot model than a a viscoelastic model. In \cite{Morency2011Acoustic-elasti}, it is shown that a poroelastic signature in time-lapse migration data from CO2 sequestration crosswell monitoring could not be explained with acoustic or elastic models. For biomechanics of cancellous bones \cite{cowin1999bone-poroelasti}, \cite{wear2010decomposition-o}, \cite{Benalla2013Dynamic-permeab}, it is conjectured that the fluid-solid interaction in the pore space (the viscodynamics) might be the signaling mechanism for bone reconstruction. At the macroscopic level, the viscodynamics of a poroelastic material is described by the dissipative and dispersive coefficients in the poroelastic wave equation with memory terms\cite{owan1997mechanotransduc}.

FWI has been applied to the one-dimensional Biot-JKD equations to recover several prorelastic parameters from numerical simulated data in a
water tank setting \cite{buchanan2011recovery-of-the}. Due to the complexity of the poroelastic equations with memory terms, FWI has not been applied to the poroelastic wave equations in dimensions higher than one until recently. Published in 2021, the first paper \cite{Qingjie-Yang2021Frequency-domai} for the 2D poroelastic FWI considers the frequency domain instead of the time domain.

The time-domain augmented Biot solver developed in \cite{Xie-Ou-Xu-2019} provides an efficient way to handle the poroelastic wave equations with memory terms. As is shown in Section \ref{sec adjoint equation}, the same solver can be used to solve the adjoint problem, whose output can be used to compute the gradient of the objective function with respect to the material properties.  

The paper is organized as follows. The poroelastic equations are stated in Section \ref{model}, where both the JKD model and the general model, together with their energy analyses are presented. The energy analysis will play an important role in the stability of solving the adjoint problem.  In Section \ref{sec adjoint equation}, the FWI problem for the augmented Biot equation is presented and the corresponding adjoint problem is derived. The algorithm for computing the poles and residues of the memory functions are given in Section~\ref{Algorithm}. 
Section~\ref{conclu} includes a conclusion of this paper and a list of possible future work. 
\section{\label{model}The mathematical model for wave propagation in poroelastic materials}
Introduce the solid velocity $\mathbf{v}$ and the relative fluid velocity $\mathbf{q}$, which are defined by
\begin{equation}
\label{Eq:velocity}
\mathbf{v}:=\frac{\partial \mathbf{u}}{\partial t} , \ \mathbf{q}:=\frac{\partial \mathbf{w}}{\partial t},
\end{equation}
where $\mathbf{u}$ is the skeleton solid displacement vector and $\mathbf{w}=\phi(\mathbf{U}-\mathbf{u})$ is  the relative motion of the fluid scaled by the porosity $\phi$ with $\mathbf{U}$ being the fluid displacement vector. In terms of these notations, the variation in fluid content is given by   $\zeta=-\nabla\cdot\mathbf{w}$.
\subsection{The constitutive relation for the plane strain of a transversely isotropic material}
Let the $z$ axis be the symmetry axis, then the stress-strain relations for a transversely isotropic elastic material reads 
\begin{equation}
\label{Eq:stree-strainelastic}
\mbox{\boldmath{$\tau$}} = \mathbf{C} \mbox{\boldmath{$\epsilon$}},
\end{equation}
where the elastic stiffness tensor $\mathbf{C}$, the stress tensor \mbox{\boldmath{$\tau$}}, and the strain tensor \mbox{\boldmath{$\epsilon$}} are defined as
\begin{eqnarray}
\label{stress-strain}
\mathbf{C}=
\left(
\begin{array}{ccc}
{c_{11}} & {c_{13}} & {0} \\
{c_{13}} & {c_{33}} & {0}  \\
{0} & {0} & {c_{55}}
\end{array}
\right),
\quad \mbox{\boldmath{$\tau$}}=(\tau_{11}\ \tau_{33}\   \tau_{13} )^T,  \quad
\mbox{\boldmath{$\epsilon$}}=(\epsilon_{11}\  \epsilon_{33}\  2\epsilon_{13}\ )^T,
\end{eqnarray}
with $\epsilon_{ij}=\frac{1}{2}\left(\partial_{i}u_{j}+\partial_{j}u_{i}\right)$, j=1,3. The constitutive relation for the poroelastic material is
\begin{eqnarray}
\label{tau}
\boldsymbol{\tau}&=&\mathbf{C} \boldsymbol{\epsilon}-\boldsymbol{\beta} p
=\mathbf{C}^u\boldsymbol{\epsilon}-M\boldsymbol{\beta}\zeta,\\
\label{p}
p&=&M\left(\zeta-\boldsymbol{\beta}^{T} \boldsymbol{\epsilon}\right),
\end{eqnarray}
where the $\mathbf{C}^u$ is the undrained elastic matrix. More precisely,
\begin{eqnarray*}
\label{coefficients}
 \mathbf{C}^u&=& \mathbf{C}+M\boldsymbol{\beta}\boldsymbol{\beta}^T,~~\boldsymbol{\beta}:= (\beta_{1}, ~\beta_{3}, ~0)^T, \\
  \beta_{1}: &=& 1-\frac{c_{11}+c_{12}+c_{13}}{3K_{s}}, ~\beta_{3}: = 1-\frac{2c_{13}+c_{33}}{3K_{s}}, \\
  M: &=& \frac{K^2_{s}}{K_{s}\left[1+\phi(K_{s}/K_{f}-1)\right]-\left(2c_{11}+c_{33}+2c_{12}+4c_{13}\right)/9},
\end{eqnarray*}
with $K_{s}$ and $K_{f}$ being the bulk modulus of the skeleton and pore fluid, respectively.

\subsection{\label{motion for Biot JKD }Equation of motions}

The equations of motion for the solid part are given by the conservation of momentum as
\begin{equation}
\label{eq:solid momentum}
\rho \frac{\partial{\mathbf{v}}}{\partial t}+\rho_{f} \frac{\partial \mathbf{{q}}}{\partial t}=\nabla\cdot \boldsymbol{\overline{\tau}},
\end{equation}
where $\rho_{s}$  is the density of constituent solid, $\rho=(1-\phi)\rho_{s}+\phi \rho_{f}$ is the bulk density of the medium and $\boldsymbol{\overline{\tau}}=
\left[ \begin{array}{cc}{\tau_{11}} & {\tau_{13}} \\ {\tau_{13}} & {\tau_{33}}\end{array}\right]
$ with the action $\nabla\cdot \boldsymbol{\overline{\tau}}:=\left(\partial_x\tau_{11}+\partial_z\tau_{13}, \ \partial_{x}\tau_{13}+\partial_z\tau_{33} \right)^T$ is the plane stress tensor.

The equations of motion for the fluid part is described by the generalized Darcy's law
\begin{equation}
\label{eq:fluidmotion Darcy}
\rho_f \frac{\partial{\mathbf{v}}}{\partial t}+{\DIAG}\left(\frac{\rho_f}{\phi} \right) \check{\boldsymbol{\alpha}} \star\frac{\partial \mathbf{q}}{\partial t}=-\nabla p,
\end{equation}
where $\star$ denotes the time-convolution operator and $\check{\boldsymbol{\alpha}}=(\check \alpha_1,\check \alpha_3)^T$ is the inverse Fourier-Laplace transform of the dynamic tortuosity $\boldsymbol{\alpha}(s)$ with $s=-i\omega$. The  Fourier-Laplace transform of a function $f(t)$ is defined as
\beq
\hat{f}(\omega):={\mathcal{L}}[f](s=-i\omega):=\frac{1}{\sqrt{2\pi}} \int_{0}^{\infty} f(t)e^{-st} dt .
\label{one-side-LT}
\eeq
For the well-known low-frequency Biot model \cite{biot1956theory-high},  the generalized Darcy's law is 
\[ 
-\pr_{x_i} p=\rho_f \pr_t v_i +  \left(\frac{\rho_f}{\phi} \right) {\alpha_\infty}_i \pr_t q_i +\left(\frac{\eta}{K_{0i}}\right)q_i\label{darcy-x}
\] 
, which has no memory term and corresponds to  
\[
\check{\alpha}_j(t)=\alpha_{\infty j} \delta(t)+\frac{\eta \phi}{K_{0j} \rho_f} H(t)\Longleftrightarrow  \alpha_j(\omega)={\alpha_\infty}_j+\frac{\eta \phi/(K_{0j} \rho_f)}{-i\omega}.
\]
Historically, to address the discrepancy between the lab measurement and the prediction by the low-frequency Biot model, many forms of $\alpha(\omega)$ have been proposed \cite{biot1956theory-of-propa} \cite{johnson1987theory-of-dynam} \cite{avellaneda1991rigorous-link-b} \cite{Pride1992Modeling-The-Dr}. Among them, the JKD model by Johnson, Koplik and Dashen  \cite{johnson1987theory-of-dynam} has been widely used in applications. The most general one is presented in \cite{avellaneda1991rigorous-link-b}, which does not rely on any specific form of the  
permeability/tortuosity function. In \cite{ou2014on-reconstructi}, an integral representation formula for $\alpha(\omega)$ was derived based on the result in \cite{avellaneda1991rigorous-link-b} and the Stieltjes function theory. 
\subsection{\label{Biot JKD }The Biot-JKD equations}
The JKD model \cite{johnson1987theory-of-dynam} is
\begin{equation}
\label{Eq:Tortuosity}
\alpha_j(s)=T_j(s):=\alpha_{\infty j}+\frac{\eta\phi}{s \kappa_{j}\rho_f}\left(1+s\frac{4 \alpha^2_{\infty j}\kappa^2_{j}\rho_f}{\eta \Lambda_j^2\phi^2}\right)^{\frac{1}{2}},\quad s:=-i\omega
\end{equation}
where
\begin{equation*}
\Lambda_{j}=\sqrt{\frac{4 \alpha_{\infty j} \kappa_{j}}{\phi \text{P}_{j}}},~j=1,3,
\end{equation*}
is the viscous characteristic length and $\text{P}_j=0.5$ the Pride numberin the $j$-th direction. For the JKD model, equation  (\ref{Eq:Tortuosity}) reads
\begin{equation}
\label{eq:HighDarcy}
\rho_f \frac{\partial{\mathbf{v}}}{\partial t}+{\DIAG}\left(\frac{\alpha_{\infty j}\rho_f}{\phi} \right) \frac{\partial{\mathbf{q}}}{\partial t}+ \nabla p=-{\DIAG}\left(\frac{\eta}{\kappa_j \sqrt{\lambda_j}} \right) \left(\partial_t+\lambda_j\right)^{1/2}\mathbf{q},~j=1,3,
\end{equation}
where $\lambda_j=\frac{\eta\phi^2\Lambda_j^2}{4\alpha^2_{\infty j}\kappa_j^2\rho_f}$.

Taking  derivative on both sides (\ref{tau}) and (\ref{p}) with resect to time $t$, and taking into account the equation of motions, the Biot-JKD equations are as follows
\begin{empheq}[left=\empheqlbrace]{align}
&\frac{\partial \boldsymbol{\tau}}{\partial t}=\frac{\partial }{\partial t}\left(\mathbf{C}^u\boldsymbol{\epsilon}-M\boldsymbol{\beta}\zeta\right),\label{partialtau}\\
&\frac{\partial p}{\partial t}=\frac{\partial }{\partial t}\left(M\left(\zeta-\boldsymbol{\beta}^{T} \boldsymbol{\epsilon}\right)\right),\label{partialp}\\
&\rho \frac{\partial{\mathbf{v}}}{\partial t}+\rho_{f} \frac{\partial \mathbf{{q}}}{\partial t}=\nabla\cdot \boldsymbol{\overline{\tau}},\label{eq:eqmomentum}\\
&\rho_f \frac{\partial{\mathbf{v}}}{\partial t}+{\DIAG}\left(m_j \right) \frac{\partial{\mathbf{q}}}{\partial t}+ \nabla p=-{\DIAG}\left(\frac{\eta}{\kappa_j \sqrt{\lambda_j}}  \left(\partial_t+\lambda_j\right)^{1/2}\right)\mathbf{q}, \label{eq:eqdarcy}
\end{empheq}
where $\left(\partial_t+\lambda_j\right)^{1/2}$ is a shift fractional derivative operator  and $m_j=\alpha_{\infty j}\rho_f/\phi >0$.
\subsection{Energy analysis of the Biot-JKD equations}
Using the Caputo fractional derivative
\begin{equation}
\label{Caputo fractional}
\frac{d^{\alpha}f(t)}{d t^{\alpha}}=\frac{1}{\Gamma(1-\alpha)}t^{-\alpha}\star\frac{\partial f}{\partial t},
\end{equation}
we can explicitly derive the shifted fractional derivative as
\begin{equation}
\label{eq:DAIntg}
\begin{aligned}
(\partial_t+\lambda_j)^{1/2}q_i=\frac{2}{\pi}\int_0^{\infty}\psi_j(y,t)dy,~j=1,3,
\end{aligned}
\end{equation}
where the auxiliary variables are defined as
\begin{equation}
\begin{aligned}
\psi_j(y,t)=\int_0^te^{-(y^2+\lambda_j)(t-\tau)}\left[\lambda_j q_j(\tau)+\partial_{\tau} q_j(\tau)\right] d\tau.
\end{aligned}
\end{equation}
It is easy to  check that the auxiliary variables satisfy the following equations
\begin{eqnarray}
\begin{aligned}
\label{eq:auxvarrho}
\left\{\begin{array}{cc}
\frac{\partial \psi_j}{\partial t}=-\left(y^2+\lambda_j\right)\psi_j+\left[\lambda_j q_j+\partial_t q_j\right],\\
\psi_j(y,0)=0,\,j=1,3.
\end{array}\right.
\end{aligned}
\end{eqnarray}
\begin{theorem}
\label{thm:JKD}
Consider the Biot-JKD equation (\ref{partialtau})-(\ref{eq:eqdarcy}), define
\begin{eqnarray}
\label{def:OriginalE1}
\mathcal{E}_{1}&=&\frac{1}{2} \int_{\mathbb{R}^{2}}\rho \mathbf{v}^{T} \mathbf{v}+ 2 \rho_{f} \mathbf{v}^{T} \mathbf{q}+\mathbf{q}^{T} {\DIAG}({m_j})\mathbf{q} ~ d x d z,\\
\label{def:OriginalE2}
\mathcal{E}_{2}&=&\frac{1}{2} \int_{\mathbb{R}^{2}}\left((\boldsymbol{\tau}+p \boldsymbol{\beta})^{T} \boldsymbol{C}^{-1}(\boldsymbol{\tau}+p \boldsymbol{\beta})+\frac{1}{M} p^{2}\right) dx d z,\\
\label{E3}
\mathcal{E}_3&=&\frac{\eta}{\pi}\int_{\mathbb{R}^{2}}\int_0^{\infty}(\mathbf{q}-\boldsymbol{\psi})^T{\DIAG}
\left(\frac{1}{\kappa_j\sqrt{\lambda_j}(y^2+2\lambda_j)}\right)(\mathbf{q}-\boldsymbol{\psi})dydxdz,
\end{eqnarray}
then $\mathcal{E}=\mathcal{E}_1+\mathcal{E}_2+\mathcal{E}_3$ is the energy function, which satisfies
\begin{equation}\label{def:DEnergy}
\begin{aligned}
\frac{d}{d t}\mathcal{E}&=
-\frac{2\eta}{\pi}\int_{\mathbb{R}^{2}}
\int_{0}^{\infty}\boldsymbol{\psi}^T
{{\DIAG}
\left(\frac{y^2+\lambda_j}{\kappa_j\sqrt{\lambda_j}(y^2+2\lambda_j)}\right)
\boldsymbol{\psi}dydxdz}\\
&\quad-\frac{2\eta}{\pi}\int_{\mathbb{R}^{2}}
\int_{0}^{\infty}\mathbf{q}^T{\DIAG}
\left(\frac{\lambda_j}{\kappa_j\sqrt{\lambda_j}(y^2+2\lambda_j)}\right)\mathbf{q}dydxdz\le0,
\end{aligned}
\end{equation}
where $\boldsymbol{\psi}=\left(\psi_1,~\psi_3\right)^T$.
\end{theorem}

\begin{proof}
First, we prove that the energy function  $\mathcal{E}$ is  positive definite. Obviously,  both $\mathcal{E}_2$ and $\mathcal{E}_3$ are  positive definite. Note that $\mathcal{E}_1$ can be expressed as
$$
\mathcal{E}_{1}=\frac{1}{2} \int_{\mathbb{R}^{2}} \widetilde{\mathbf{V}}_j^T \mathcal{M}_j\widetilde{\mathbf{V}}_j  d x d z,\quad \mathcal{M}_j=\left[ \begin{array}{cc}{\rho} & {\rho_f} \\ {\rho_f} & {m_j}\end{array}\right],\quad j=1,3,
$$
with $\widetilde{\mathbf{V}}_j=\left(v_j,~q_j\right)^T$.  Since  $\det(\mathcal{M}_j)>0$, $\mathcal{E}_1$ is positive definite.
Next, we prove that (\ref{def:DEnergy}) holds.  Multiplying (\ref{eq:solid momentum}) with $\mathbf{v}^T$ and integrating on $\mathbb{R}^2$, we obtain
\begin{equation}
\label{Mcn}
0=\int_{\mathbb{R}^{2}}\left(\rho \mathbf{v}^{T} \frac{\partial \mathbf{v}}{\partial t}+\rho_{f} \mathbf{v}^{T} \frac{\partial \mathbf{q}}{\partial t}-\mathbf{v}^{T}\nabla \cdot \boldsymbol{\overline{\tau}}\right) d x d z:=\sum_{k=1}^3I_k(\mathbf{v}).
\end{equation}
Clearly, it holds that
\begin{equation*}
I_1(\mathbf{v})=\int_{\mathbb{R}^{2}} \rho \mathbf{v}^{T} \frac{\partial \mathbf{v}}{\partial t} d x d z= \frac{1}{2}  \frac{d}{d t} \int_{\mathbb{R}^{2}} \rho\mathbf{v}^{T} \mathbf{v}d x d z.
\end{equation*}
Integrating by part and noticing  that $\mathbf{v}=\partial_t\mathbf{u}$, we have
\begin{equation*}
\begin{aligned}
I_3(\mathbf{v})&=-\int_{\mathbb{R}^{2}} \mathbf{v}^{T}(\nabla \cdot \boldsymbol{\overline{\tau}}) d x d z =\int_{\mathbb{R}^{2}} \boldsymbol{\tau}^{T} \frac{\partial \boldsymbol{\epsilon}}{\partial t} d x d z\\
&=\int_{\mathbb{R}^{2}} \boldsymbol{\tau}^{T}\boldsymbol{C}^{-1}\left( \frac{\partial \boldsymbol{\tau}}{\partial t}+\boldsymbol{\beta} \frac{\partial p}{\partial t}\right) d x d z \\
&= \frac{1}{2} \frac{d}{d t}\int_{\mathbb{R}^{2}} \boldsymbol{\tau}^{T} \boldsymbol{C}^{-1} \boldsymbol{\tau} d x d z+\int_{\mathbb{R}^{2}} \boldsymbol{\tau}^{T} \mathbf{C}^{-1} \boldsymbol{\beta} \frac{\partial p}{\partial t} d x d z \\
&= \frac{1}{2} \frac{d}{d t}\int_{\mathbb{R}^{2}}\left(\boldsymbol{\tau}^{T} \mathbf{C}^{-1} \boldsymbol{\tau}+2 \boldsymbol{\tau}^{T} \mathbf{C}^{-1} \boldsymbol{ \beta} p\right) d x d z-\int_{\mathbb{R}^{2}}\left(\frac{\partial \boldsymbol{\tau}}{\partial t}\right)^{T} \mathbf{C}^{-1} \boldsymbol{\beta} p d x d z,
\end{aligned}
\end{equation*}
where (\ref{tau}) is used in the third equation.

Multiplying the generalized Darcy's law (\ref{eq:eqdarcy}) with $\mathbf{q}^T$ and integrating on $\mathbb{R}^2$ lead to 
\begin{equation}
\begin{aligned}
\label{MDc}
0=&\int_{\mathbb{R}^{2}}\left[\rho_f \mathbf{q}^T \frac{\partial \mathbf{v}}{\partial t}+ \mathbf{q}^T {\DIAG}(m_j)\frac{\partial \mathbf{q}}{\partial t} \right.\\
&\left.+\mathbf{q}^T \nabla p +\mathbf{q}^T{\DIAG}\left(\gamma_j \left(\partial_t+\lambda_j\right)^{1/2}\right)\mathbf{q}\right]dxdz
:=\sum_{k=1}^4L_k(\mathbf{q}),
\end{aligned}
\end{equation}
where $\gamma_j=\frac{\eta}{\kappa_j \sqrt{\lambda_j}}$. We rewrite $L_2$ and $L_3$ as follows.
\begin{equation*}
L_2(\mathbf{q})=\int_{\mathbb{R}^{2}} \mathbf{q}^{T} \operatorname{diag}\left(m_{ i}\right) \frac{\partial \mathbf{q}}{\partial t} d x d z=
\frac{1}{2} \frac{d}{d t} \int_{\mathbb{R}^{2}} \mathbf{q}^{T} \operatorname{diag}\left(m_{i}\right) \mathbf{q} d x d z.
\end{equation*}
Integrating by part and using $\nabla \cdot \mathbf{q}=-\partial_t\zeta$ as well as (\ref{p}), we derive
\begin{equation*}
\begin{aligned}
L_3(\mathbf{q})&=\int_{\mathbb{R}^{2}} \mathbf{q}^{T} \nabla p d x d z =-\int_{\mathbb{R}^{2}} p \nabla \cdot \mathbf{q} dxdz
=\int_{\mathbb{R}^{2}} p\left(\frac{1}{M} \frac{\partial p}{\partial t}+\boldsymbol{\beta}^{T} \frac{\partial \boldsymbol{\epsilon}}{\partial t}\right) d x d z \\
&=\frac{1}{2} \frac{d}{d t}  \int \frac{1}{M} p^{2} d x d z+\int_{\mathbb{R}^{2}} p\boldsymbol{\beta}^{T} \mathbf{C}^{-1} \frac{\partial \boldsymbol{ \tau}}{\partial t}  d x d z+\int_{\mathbb{R}^{2}} \boldsymbol{\beta}^{T} \mathbf{C}^{-1} \boldsymbol{\beta} p \frac{\partial p}{\partial t} d x d z \\
&=\frac{1}{2} \frac{d}{d t}  \int \frac{1}{M} p^{2} d x d z+\int_{\mathbb{R}^{2}} \boldsymbol{\beta}^{T} \mathbf{C}^{-1} \frac{\partial \boldsymbol{\tau}}{\partial t} p dxdz
+\frac{1}{2} \frac{d}{d t}  \int_{\mathbb{R}^{2}} \boldsymbol{\beta}^{T} \mathbf{C}^{-1} \boldsymbol{\beta} p^{2} dxdz.
\end{aligned}
\end{equation*}
Adding (\ref{Mcn}) and (\ref{MDc}) and noting that the negative term in $I_3$ is cancelled by the second term in $L_3$ because of the symmetry of $\mathbf{C}$, we obtain
\begin{equation*}
\label{De12}
\begin{aligned}
0&=\frac{1}{2} \frac{d}{d t} \int_{\mathbb{R}^{2}} \left[\rho\mathbf{v}^{T} \mathbf{v}+ \mathbf{q}^{T} \DIAG(m_j) \mathbf{q} +2\rho_f\mathbf{v}^{T} \mathbf{q}\right]d x d z\\
&\quad + \frac{1}{2}\frac{d}{d t} \int_{\mathbb{R}^{2}} \left[\frac{1}{M} p^{2}+\boldsymbol{\beta}^{T}\mathbf{C}^{-1} \boldsymbol{\beta} p^{2}
+\left(\boldsymbol{\tau}^{T} \mathbf{C}^{-1} \boldsymbol{\tau}+2 \boldsymbol{\tau}^{T} \mathbf{C}^{-1} \boldsymbol{\beta} p\right)\right] d x d z\\
&\quad
+\int_{\mathbb{R}^{2}}\mathbf{q}^T{\DIAG}\left(\gamma_j \left(\partial_t+\lambda_j\right)^{1/2}\right)\mathbf{q}dxdz.
\end{aligned}
\end{equation*}
In terms of $\mathcal{E}_1$ and $\mathcal{E}_2$ and \eqref{eq:DAIntg}, the equation above can be expressed as
\begin{equation}
\begin{aligned}
\label{eq:dE121}
\frac{d}{dt}(\mathcal{E}_1+\mathcal{E}_2)&
=-\int_{\mathbb{R}^{2}}\mathbf{q}^T{\DIAG}\left(\gamma_j \left(\partial_t+\lambda_j\right)^{1/2}\right)\mathbf{q}dxdz\\
&=-\frac{2\eta}{\pi}\int_{\mathbb{R}^{2}}\int_{0}^{\infty}\mathbf{q}^T{\DIAG}
\left(\frac{1}{\kappa_j\sqrt{\lambda_j}}\right)\boldsymbol{\psi}(y,t) dy\,dxdz.
\end{aligned}
\end{equation}

To estimate the right hand side of (\ref{eq:dE121}), we multiply (\ref{eq:auxvarrho}) with $\mathbf{q}^T$ and $\boldsymbol{\psi}^T$ respectively, to obtain
\begin{numcases}{}
\label{eq:DerivAugd}
\mathbf{q}^T \frac{\partial \boldsymbol{\psi}}{\partial t}- \mathbf{q}^T \frac{\partial \mathbf{q}}{\partial t}+\mathbf{q}^T {\DIAG}({y^2+\lambda}_j) \boldsymbol{\psi}-\mathbf{q}^T{\DIAG}(\lambda_j)\mathbf{q}=0, \\
\label{eq:AugInitiald}
\boldsymbol{\psi}^T \frac{\partial \boldsymbol{\psi}}{\partial t}- \boldsymbol{\psi}^T \frac{\partial \mathbf{q}}{\partial t}+\boldsymbol{\psi}^T {\DIAG}({y^2+\lambda}_j) \boldsymbol{\psi}-\boldsymbol{\psi}^T{\DIAG}(\lambda_j)\mathbf{q}=0,\, j=1,3.
\end{numcases}
Subtracting (\ref{eq:AugInitiald}) from (\ref{eq:DerivAugd}) reveals that

\begin{equation*}
(\mathbf{q}-\boldsymbol{\psi})^T\frac{\partial}{\partial t}(\boldsymbol{\psi}-\mathbf{q})+(\mathbf{q}-\boldsymbol{\psi})^T{\DIAG}(y^2+\lambda_j)\boldsymbol{\psi}
-(\mathbf{q}-\boldsymbol{\psi})^T{\DIAG}(\lambda_j)\mathbf{q}=0,
\end{equation*}
which indicates
\begin{equation*}
\mathbf{q}^T{\DIAG}(y^2+2\lambda_j)\boldsymbol{\psi}
=\frac{1}{2}\frac{d}{dt}(\mathbf{q}-\boldsymbol{\psi})^T(\mathbf{q}-\boldsymbol{\psi})
+\boldsymbol{\psi}^T{\DIAG}(y^2+\lambda_j)\boldsymbol{\psi}
+\mathbf{q}^T{\DIAG}(\lambda_j)\mathbf{q}.
\end{equation*}
The theorem is proved by injecting the equation into (\ref{eq:dE121}).
\end{proof}
\subsection{The augmented Biot equations}
According to \cite{ou2014on-reconstructi}, the tortuosity function of all porous material can be expressed as
\begin{equation*}
\label{eq:TorJKD}
\alpha_j(\omega)=\frac{\eta \phi}{\rho_{f} \kappa_{j}}\left(\frac{1}{s}\right)+\int_{0}^{\theta_j} \frac{d \sigma_{j}(t)}{1+s t}=:T_{j}(s),\quad s=-i\omega,\quad j=1,3,
\end{equation*}
where $d\sigma_j$, $j=1,3$, is a probability measure, and the upper limit $\theta_j>0$ is related to the inverse of the smallest eigenvalue of the Stokes equations posed in the pore space of the porous material \cite{avellaneda1991rigorous-link-b}. This is derived from the result in \cite{avellaneda1991rigorous-link-b} and the Stieltjes function theory. 

Define the following function
\begin{equation}
\label{eq:NewFunD}
D_{j}(s) :=T_{j}(s)-\frac{ a_{j}}{s}=\int_{0}^{\theta_{1}} \frac{d \sigma_{j}(t)}{1+s t},~a_{j} :=\frac{\eta \phi}{\rho_{f} \kappa_{j}},
\end{equation}
then $D_j(s)$  is a Stieltjes function and  hence can be approximated as \cite{gelfgren1978rational}
\begin{equation*}
\label{eq:StielAppro}
D_{j}(s)=\int_{0}^{\theta_{j}} \frac{d \sigma_{j}(t)}{1+s t} \approx \alpha_{\infty_j}
+\sum_{k=1}^{N_j} \frac{r_{k}^{j}}{s-\vartheta_{k}^{j}},~ s \in \mathbb{C} \backslash\left(-\infty,-\theta_j^{-1}\right],~ j=1,3,
\end{equation*}
where $r_k^{j}>0$ are the residues  and $\vartheta_k^{j}<0$ are the poles. For simplicity, we assume $\B{q}(\bx,0)=\B{0}$. By the Fourier-Laplace transform and (\ref{eq:NewFunD}), with $\B{q}(0)=0$, we have
\begin{equation*}
\begin{aligned}
\label{eq:LapCon}
\mathcal{L}\left[\alpha_j \star \frac{\partial q_{j}}{\partial t}\right](s) &
=\alpha_{\infty_j} s \hat{q}_{j}+\left(a_j+\sum_{k=1}^{N_j} r_{k}^{j}\right) \hat{q}_{j}+\sum_{k=1}^{N_j} r_{k}^{j} \vartheta_{k}^{j} \frac{\hat{q}_{j}}{s-\vartheta_{k}^{j}}, \quad j=1,3.
\end{aligned}
\end{equation*}
Applying the inverse Fourier-Laplace transform gives
\begin{equation*}
\begin{aligned}
\label{eq:InvLapCon}
\left(\alpha_j \star \frac{\partial q_{j}}{\partial t}\right)(\boldsymbol{x}, t)\approx \alpha_{\infty_j}\frac{\partial q_{j}}{\partial t}+
\left(a_j+\sum_{k=1}^{N_j} r_{k}^{j}\right) q_{j}-\sum_{k=1}^{N_j} r_{k}^{j}\left(-\vartheta_{k}^{j}\right) e^{\vartheta_{k}^{j} t} \star q_{j}.
\end{aligned}
\end{equation*}
Introduce the  auxiliary variables $\Theta^{j}_k$, with $k=1,...N_j$
\beq
\Theta_k^{j}(\bx,t):= (-\vartheta_k^{j}) \exp({\vartheta_k^{j}\, t}) \star  q_j,~j=1,3,
\label{theta_def}
\eeq
then the following holds
\begin{numcases}{}
\label{eq:Deriv Aug}
\partial_t\Theta_k^{j}(\bx,t)=\vartheta_k^{j} \Theta_k^{j}(\bx,t)-\vartheta_k^{j} q_j(\bx,t), \\
\label{eq:Aug Initial}
\Theta_k^{j}(\bx,0)=0,\, j=1,3.
\end{numcases}

The constants $r_k^{j}$ and $\vartheta_k^{j}$, $k=1,\cdots,N_j$ are related to the pole-residue approximation for $\alpha_j(s)$ for the frequency range relevant to the point source spectral content; they can be computed with very high accuracy with the algorithms presented in \cite{ou2018augmented}, which is summarized in Section \ref{Algorithm}. The integral representation formula of $\alpha_j(s)$ guarantees that $r_k^{j}>0$ and $\vartheta_k^{j}<0$, $j=1,3$.
\vspace{0.1in}
With these notations, the Darcy's law (\ref{eq:fluidmotion Darcy}) can be written as
\begin{equation}
\label{Eq:HighDarcyTime}
\begin{aligned}
-\nabla p =\rho_f \partial_{t}\mathbf{v}+{\DIAG}(m_j)\partial_{t}\mathbf{q}
+{\DIAG}\left(\frac{\eta}{\kappa_j}+\frac{\rho_f}{\phi}\sum_{k=1}^{N_j} r^{j}_k\right)\mathbf{q}
-\frac{\rho_f}{\phi} \begin{pmatrix}
\sum_{k=1}^{N_1} r_k^1 \Theta_k^1\\
\sum_{k=1}^{N_3} r_k^3 \Theta_k^3
\end{pmatrix}.
%\sum_{k=1}^{N_j}{\DIAG}\left(r^{j}_k\right)\boldsymbol{\Theta}_k(t),
\end{aligned}
\end{equation}
%where $\boldsymbol{\Theta}_k(t)=\left(\Theta^{1}_k,\Theta^{3}_k\right)^T$.

For convenience, we introduce the following notation

\begin{eqnarray*}
{D}=\left[
\begin{array}{cccc}
{\partial_x} & {0}   & {\partial_z} & 0\\
{0}    &{\partial_z} & {\partial_x} &0 \\
{0   } & {0}   & 0   &\partial_x\\
{0   } & {0}   & 0   &\partial_z\\
\end{array}
\right]
=\left[
\begin{array}{cccc}
1 & {0}   & 0 & 0\\
{0}    &0 & 1 &0 \\
{0   } & {0}   & 0   &1\\
{0   } & {0}   & 0   &0\\
\end{array}
\right]\partial_x
+\left[
\begin{array}{cccc}
0 & {0}   &1 & 0\\
{0}    &1 & 0 &0 \\
{0   } & {0}   & 0   &0\\
{0   } & {0}   & 0   &1\\
\end{array}
\right]\partial_z.
\end{eqnarray*}
Define
$\boldsymbol{\sigma}=\left(\boldsymbol{\tau}^T,~-p \right)^{T} $, $\mathbf{e}=\left(\boldsymbol{\epsilon}^T,~- \zeta \right)^{T} $ , and $\mathbf{V}=\left(v_1,~  v_3,~  q_1,~ q_3 \right)^{T} $,  then (\ref{partialtau}) and (\ref{partialp}) can be rewritten as
\begin{equation}
\label{eq:DtauM}
\frac{\partial \boldsymbol{\sigma}}{\partial t}=\mathbf{E}D^{T}\mathbf{V}, \quad   \mathbf{E}=\left[
\begin{array}{cc}
{\mathbf{C}^u}        & M\boldsymbol{\beta}\\
M\boldsymbol{\beta}^T &M\\
\end{array}
\right].
\end{equation}
Replacing (\ref{eq:eqdarcy}) with (\ref{Eq:HighDarcyTime}) and combining (\ref{eq:eqmomentum}), we have
\begin{equation}
\label{eq:matrixV}
M_v\frac{\partial \mathbf{V}}{\partial t}=D \mathbf{\boldsymbol \sigma}-M_r \mathbf{V}+R_x\mathbf{\Theta}^x+R_z\mathbf{\Theta}^z,
\end{equation}
where $\boldsymbol{\Theta}^1=(\Theta^1_1,...,\Theta^1_{N_1})^T$ and $\boldsymbol{\Theta}^3=(\Theta^3_1,...,\Theta^3_{N_3})^T$,
\begin{eqnarray*}
{M_v}=\left[
\begin{array}{cccc}
{\rho} & {0}   & {\rho_f} & 0\\
{0}    &{\rho} & {0}      &\rho_f\\
{\rho_f} & {0}   & m_1&0\\
0        & \rho_f& 0  &m_3\\
\end{array}
\right],
~~~~~{M_r}=\left[
\begin{array}{cccc}
{0} & {0}   & {0} & 0\\
{0}    &{0} & {0}      &0\\
{0} & {0}   & n_1&0\\
0        & 0& 0  &n_3\\
\end{array}
\right],
\end{eqnarray*}
\begin{eqnarray*}
{R_x}=\frac{\rho_f}{\phi} \left[
\begin{array}{ccccccccc}
0 & {0}  & 0&... & 0\\
{0} &0 & {0}&...  &0\\
r^x_1 & r^x_2 &r^x_3&...&r^x_{N_x}\\
0 &0  &0     & ...  &0\\
\end{array}
\right],
~~~~~~{R_z}=\frac{\rho_f}{\phi}\left[
\begin{array}{ccccccccc}
0 & {0}  & 0&... & 0\\
{0} &0 & {0}&...  &0\\
0 &0  &0     & ...  &0\\
r^z_1 & r^z_2 &r^z_3&...&r^z_{N_z}\\
\end{array}
\right],
\end{eqnarray*}
with  $n_j=\frac{\eta}{\kappa_j}+\frac{\rho_f}{\phi}\sum_{k=1}^{N_j}r_k^{j}>0$.

Denote $\mathbf{W}=\left(v_1, v_3, q_1,q_3,\tau_{xx},\tau_{zz},\tau_{xz}, -p,\Theta_1^1,\cdots,\Theta^1_{N_1},\Theta^3_1,\cdots,\Theta^3_{N_3} \right)^{T}$. Equations (\ref{eq:DtauM}) and (\ref{eq:matrixV}) lead to the following augmented Biot equations
\begin{eqnarray*}
\begin{aligned}
\left[
\begin{array}{cccc}
M_v & \mathbf{0}   & {\mathbf{0}} & \mathbf{0}\\
\mathbf{0}   &\mathbf{E}^{-1} & \mathbf{0} &\mathbf{0} \\
\mathbf{0} & \mathbf{0}   &\mathbf{I}   &\mathbf{0}\\
\mathbf{0} & \mathbf{0}  & \mathbf{0}  &\mathbf{I}\\
\end{array}
\right]
 \left[
\begin{array}{cccc}
\mathbf{V}\\
\mbox{\boldmath{$\sigma$}}\\
{\mathbf{\Theta}^1}\\
\mathbf{\Theta}^3\\
\end{array}
\right]_t
&=\left[
\begin{array}{cccc}
\mathbf{0}& D& \mathbf{0}&\mathbf{0}\\
D^{T}&\mathbf{0}&\mathbf{0}&\mathbf{0}\\
\mathbf{0} & \mathbf{0}   &\mathbf{0}   &\mathbf{0}\\
\mathbf{0} & \mathbf{0}  & \mathbf{0}  &\mathbf{0}\\
\end{array}
\right]
\left[
\begin{array}{cccc}
\mathbf{V}\\
\mbox{\boldmath{$\sigma$}}\\
{\mathbf{\Theta}^1}\\
\mathbf{\Theta}^3\\
\end{array}
\right]\\
&\quad+
\left[
\begin{array}{cccc}
-M_r&\mathbf{0}&R_x&R_z\\
\mathbf{0} & \mathbf{0}   &\mathbf{0}   &\mathbf{0}\\
-\mathcal{Q}_1 & \mathbf{0}   &\mathcal{P}_x   &\mathbf{0}\\
-\mathcal{Q}_3 & \mathbf{0}  & \mathbf{0}  &\mathcal{P}_z\\
\end{array}
\right]
\left[
\begin{array}{cccc}
\mathbf{V}\\
\mbox{\boldmath{$\sigma$}}\\
{\mathbf{\Theta}^1}\\
\mathbf{\Theta}^3\\
\end{array}
\right],
\end{aligned}
\end{eqnarray*}
or equivalently
\begin{equation}
\label{eq:Fdireqn}
\mathcal{N}_1(\mathbf{M})\partial_t \mathbf{W} - B_1(\bigtriangledown)\mathbf{W}-\mathcal{N}_2(\mathbf{M})\mathbf{W}-\mathbf{S}=\mathbf{0},
\end{equation}
where $\mathbf{M}$ is the set of the parameters, $\mathbf{S}=\left(\mathbf{f}_v, \mathbf{C}^{-1}\mathbf{f}_{\tau}, 0,\cdots,0\right)^{T}$ is the external source and
\begin{eqnarray*}
{\mathcal{Q}_1}=\left[
\begin{array}{cccc}
{0   } & {0}   & \vartheta^1_1   &0\\
{0   } & {0}   & \vartheta^1_2   &0\\
... & ...   & ...        &...\\
{0   } & {0}   & \vartheta^1_{N_1}   &0\\
\end{array}
\right],
~~~{\mathcal{Q}_3}=\left[
\begin{array}{cccc}
{0   } & {0}   &0& \vartheta^3_1   \\
{0   } & {0}   &0& \vartheta^3_2   \\
... & ...   & ...        &...\\
{0   } & {0}   &0 & \vartheta^3_{N_3}   \\
\end{array}
\right],
\end{eqnarray*}
\begin{eqnarray*}
\mathcal{P}_x=\left[
\begin{array}{cccc}
{  \vartheta^1_1 } & {0}   & 0   &0\\
{0   } & {\vartheta^x_2}   & 0  &0\\
... & ...   & ...        &...\\
{0   } & {0}   &0   & \vartheta^1_{N_1}\\
\end{array}
\right],
~~~{\mathcal{P}_z}=\left[
\begin{array}{cccc}
{\vartheta^3_1   } & {0}   &0& 0   \\
{0   } & {\vartheta^3_1 }   &0& 0   \\
... & ...   & ...        &...\\
{0   } & {0}   &0 & \vartheta^z_{N_3}   \\
\end{array}
\right].
\end{eqnarray*}
Note that \eqref{eq:Fdireqn} can be expressed in the following equivalent form
\begin{equation}
\label{eq:direqn}
\partial_t \mathbf{W}(\bM,\bx,t) - \mathcal{N}^{-1}_1(\mathbf{M})B_1(\bigtriangledown)\mathbf{W}-\mathcal{N}^{-1}_1(\mathbf{M})
\mathcal{N}_2(\mathbf{M})\mathbf{W}-\mathcal{N}^{-1}_1(\mathbf{M})\mathbf{S}=\mathbf{0}.
\end{equation}
\subsection{Energy analysis of the augmented Biot equations}
In this section, we conduct the energy analysis of the augmented Biot equations \eqref{eq:direqn} with $\B{S}=0$ and $\B{q}(\bx,0)=0$.
\begin{theorem}
Consider the augmented Biot equation (\ref{eq:direqn}) and define $\mathcal{E}_1$ as (\ref{def:OriginalE1}) and $\mathcal{E}_2$ as (\ref{def:OriginalE2}).  Introduce
\begin{equation*}
\label{FastE3}
\mathcal{E}'_3=  \frac{1}{2} \sum_{j=1,3}\sum_{k=1}^{N_j}\int_{\mathbb{R}^2} \left(\frac{\rho_f r_k^j}{\phi (-\vartheta_k^j)}\right)(\varphi_k^j-q_j)^2 \,dxdz,
\end{equation*}
with $\varphi_k^{j}=\exp({\vartheta_k^{j}\, t}) \star  \frac{\partial q_j}{\partial t}$.
Then the function
\begin{equation*}
\label{def:OriginalEn}
\mathcal{E}'=\mathcal{E}_1+\mathcal{E}_2+\mathcal{E}'_3
\end{equation*}
is the energy function, which satisfies
\begin{equation*}
\frac{d}{d t}\mathcal{E}'=-\int_{\mathbb{R}^2} \mathbf{q}^T {\DIAG}\left(\frac{\eta}{\kappa_j}\right) \mathbf{q}-\sum_{j=1,3}\sum_{k=1}^{N_j} (\varphi_k^j)^2\left(\frac{\rho_f}{\phi}r_k^j\right)\, dxdz\le0.
\end{equation*}
\end{theorem}
\begin{proof}
Note that
\begin{equation*}
\begin{aligned}
\mathcal{L}\left[\alpha_{j} \star \frac{\partial q_{j}}{\partial t}\right](s)&=\alpha_{\infty_{j}} s \hat{q}_{j}+\left(a_{j}+\sum_{k=1}^{N_{j}} r^j_{k}\right) \hat{q}_{j}+\sum_{k=1}^{N_{j}} r^j_{k} \vartheta_{k} \frac{\hat{q}_{j}}{s-\vartheta_{k}} \\
&=\alpha_{\infty_{j}} s \hat{q}_{j}+a_{j} \hat{q}_{j}+\sum_{k=1}^{N_{j}} \frac{r^j_{k}}{s-\vartheta_{k}}\left(s \hat{q}_{j}\right),
\end{aligned}
\end{equation*}
from which we have
\begin{equation}
\label{Eq:QtDarc}
\begin{aligned}
-\nabla p =\rho_f \partial_{t}\mathbf{v}+{\DIAG}(m_j)\partial_{t}\mathbf{q}
+{\DIAG}\left(\frac{\eta}{\kappa_j}\right)\mathbf{q}
+
\begin{pmatrix}
\sum_{k=1}^{N_1} \frac{\rho_f}{\phi}r^{1}_k  {\varphi}^1_k \\
\sum_{k=1}^{N_3} \frac{\rho_f}{\phi}r^{3}_k  {\varphi}^3_k 
\end{pmatrix}
\end{aligned}
\end{equation}
where
\begin{equation*}
\varphi_k^{j}=\exp({\vartheta_k^{j}\, t}) \star  \frac{\partial q_j}{\partial t}.
%,~ \boldsymbol{\varphi}_k=(\varphi_k^x,~\varphi_k^z)^T.
\end{equation*}
By straightforward computing, we know
\begin{numcases}{}
\label{eq:Deriv Aug1}
\frac{\partial \varphi_k^{j}}{\partial t}(\bx,t)=\vartheta_k^{j} \varphi_k^{j}(\bx,t)+ \frac{\partial q_{j}}{\partial t}, \\
\label{eq:Aug Initial1}
\varphi_k^{j}(\bx,0)=0,\, j=1,3.
\end{numcases}
Multiplying (\ref{Eq:QtDarc}) with $\mathbf{q}^T$ yields
\begin{equation*}
\begin{aligned}
\label{eq:MDqt1}
&\rho_f \mathbf{q}^T \frac{\partial \mathbf{v}}{\partial t}+ \mathbf{q}^T {\DIAG}(m_j)\frac{\partial \mathbf{q}}{\partial t}
+\mathbf{q}^T \nabla p\\
=&- \mathbf{q}^T {\DIAG}\left(\frac{\eta}{\kappa_j}\right) \mathbf{q}
-\mathbf{q}^T \begin{pmatrix}
\sum_{k=1}^{N_1} \frac{\rho_f}{\phi}r^{1}_k  {\varphi}^1_k \\
\sum_{k=1}^{N_3} \frac{\rho_f}{\phi}r^{3}_k  {\varphi}^3_k 
\end{pmatrix}
%-\sum_{k=1}^N\mathbf{q}^T
%{\DIAG}\left(\frac{\rho_f}{\phi}r^{j}_k\right)\boldsymbol {\varphi}_kdxdz.
\end{aligned}
\end{equation*}
Similar to the proof of theorem \ref{thm:JKD}, we can derive
\begin{equation}
\frac{d}{dt}(\mathcal{E}_1+\mathcal{E}_2)=\int_{\mathbb{R}^2} - \mathbf{q}^T {\DIAG}\left(\frac{\eta}{\kappa_j}\right) \mathbf{q}
- \mathbf{q}^T \begin{pmatrix}
\sum_{k=1}^{N_1} \frac{\rho_f}{\phi}r^{1}_k  {\varphi}^1_k \\
\sum_{k=1}^{N_3} \frac{\rho_f}{\phi}r^{3}_k  {\varphi}^3_k 
\end{pmatrix}
%-\sum_{k=1}^N\mathbf{q}^T
%{\DIAG}\left(\frac{\rho_f}{\phi}r^{j}_k\right)\boldsymbol {\varphi}_k
dxdz.
\label{aug_energy_1_2}
\end{equation}
To calculate the second integral in the equation above, we multiply (\ref{eq:Deriv Aug1}) with $q_j$ and ${\varphi}^j_k$ to obtain
\begin{numcases}{}
\label{eq:Deriv Aug11}
{q_j} \partial_t {\varphi}_k^j(\bx,t)- {q}_j  {\vartheta}^j_k  {\varphi}^j_k(\bx,t)- {q}_j \partial_t {q}_j=0, \\
\label{eq:Aug Initialbis}
{\varphi_k^j} \partial_t {\varphi}_k^j(\bx,t)- {\varphi_k^j}  {\vartheta}^j_k  {\varphi}^j_k(\bx,t)- {\varphi_k^j}\partial_t {q}_j=0.
\end{numcases}
By subtracting (\ref{eq:Aug Initialbis}) from (\ref{eq:Deriv Aug11}) and rearranging terms, the following equation holds for every fixed  $j=1,3$ and $k=1,\dots, N_j$, 
\begin{equation*}
-q_j \vartheta^j_k \varphi_k^j=\frac{1}{2}\pr_t(\varphi_k^j-q_j)^2 
-\vartheta_k^j(\varphi_k^j)^2 \Longrightarrow -q_j \left(\frac{\rho_f}{\phi}\right)r_k^j\varphi_k^j=\frac{1}{2}\left(\frac{\rho_f r_k^j}{\phi \vartheta_k^j}\right)\pr_t(\varphi_k^j-q_j)^2 -(\varphi_k^j)^2\left(\frac{\rho_f}{\phi}r_k^j\right)
\end{equation*}
Summing over $k$ and noting that $\vartheta_k^j<0$ and integrating in $\mathbb{R}^2$, we conclude that
\begin{equation*}
\begin{aligned}
\int_{\mathbb{R}^2} - \mathbf{q}^T \begin{pmatrix}
\sum_{k=1}^{N_1} \frac{\rho_f}{\phi}r^{1}_k  {\varphi}^1_k \\
\sum_{k=1}^{N_3} \frac{\rho_f}{\phi}r^{3}_k  {\varphi}^3_k 
\end{pmatrix}\, dxdz
&=\int_{\mathbb{R}^2}   \sum_{j=1,3}\sum_{k=1}^{N_j} \frac{1}{2}\left(\frac{\rho_f r_k^j}{\phi \vartheta_k^j}\right)\pr_t(\varphi_k^j-q_j)^2 -(\varphi_k^j)^2\left(\frac{\rho_f}{\phi}r_k^j\right) \,dxdz\\
&=-\frac{d }{dt}\mathcal{E}_3'-\int_{\mathbb{R}^2} \sum_{j=1,3}\sum_{k=1}^{N_j} (\varphi_k^j)^2\left(\frac{\rho_f}{\phi}r_k^j\right)\, dxdz.
\end{aligned}
\end{equation*}
The theorem is proved by substituting the above result into \eqref{aug_energy_1_2}.
\end{proof}
\section{Full waveform inversion of the augmented Biot equations}\label{sec adjoint equation}
The inverse problem is defined as a constrained minimization of the misfit function $\chi(\B{M},\B{W})$ with an 
evolution PDE defining relation between the sought property field $\B{M}$ and the observed wave field $\B{W}$, see 
\eqref{eq:optdef} below. A gradient type  minimization algorithm for this constrained problem takes advantage of 
the adjoint problem which can be solved with the same solver as the forward problem. Therefore,
in this section, we derive the adjoint problem for the FWI of the augmented Biot equation~\eqref{eq:direqn} with fixed $N_1$ and $N_3$. 

We introduce the notation used in this section. Note that in this section, $\mathcal{L} $ denotes a fuctional, instead of the Fourier-Laplace transform. The Frech\'{e}t derivative with respect to $\B{u}$ of a functional $\mathcal{L}(\B{u})$ acting on a suitable test function $\delta\B{u}$ is denoted by 
$d_\B{u}\mathcal{L}(\B{u})[\delta \B{u}]$. 
The gradient with respect to $\B{u}\in\R^p$ of a function $f:\R^p\to\R$ is denoted by $\prg_\B{u}f$ and
represented by a row vector. 
The gradient of a vector-valued function $\B{F}:\R^p\to\R^q$ with respect to $\B{u}\in\R^p$  
is represented by a $q\times p$ matrix defined as
\[
[\prg_{\B{u}} \B{F}]_{ij}=\frac{\partial F_i}{\partial u_j},\, i=1,\dots,q,\, j=1,\dots,p\,.
\]
In terms of this notation, the usual product rule 
for vector-valued functions 
$\B{F},\B{G}:\B{u}\in\R^p\mapsto\R^q$ represented
by column vectors  is expressed
\[
[\prg_{\B{u}} (\B{F}^T\B{G})]_{j1}=\sum_{k=1}^q \frac{\pr F_k}{\pr u_j} G_k +F_k \frac{\pr G_k}{\pr u_j}\,,
\;\;\mbox{or in matrix notation}\;
\prg_\B{u}(\B{F}^T\B{G})= (\prg_\B{u}\B{F})^T\B{G} + \B{F}^T \prg_\B{u}\B{G}\,,
\]
and the chain rule for 
$\B{u}:\B{m}\in\R^n\mapsto\R^n$ reads
\[
[\prg_{\B{m}}\B{F}(\B{u}(\B{m}))]_{ij}=\sum_{k=1}^n \frac{\partial F_i}{\partial u_k} \frac{\pr u_k}{\pr m_j}\,,
\;\;\mbox{or in matrix notation}\;
\prg_{\B{m}}\B{F}(\B{u}(\B{m}))=\prg_{\B{u}} \B{F}\prg_{\B{m}} \B{u}\,.
\]
%
% \beqas
% (\prg_{\B{u}} (\B{F}\cdot\B{G}))_{j1}=\sum_{k=1}^m G_k \frac{\pr F_k}{\pr u_j}+F_k \frac{\pr G_k}{\pr u_j},\\
% (\prg_\sbM\B{F}(\B{u}(\B{M})))_{ij}=\sum_{k=1}^n \frac{\partial F_i}{\partial u_k} \frac{\pr u_k}{\pr m_j},
% \eeqas
% or expressed in the matrix form
% \beqas
% \prg_\B{u}(\B{F}\cdot \B{G})=\B{G}^T [\prg_\B{u}\B{F}]+\B{F}^T [\prg_\B{u}\B{G}], \\
% \prg_{\B{m}}(\B{F}(\B{u}(\B{M})))=[\prg_{\B{u}} \B{F}][\prg_{\B{m}} \B{u}].
% \eeqas
The usual subscript notation for the partial derivative $f_x:=\frac{\partial f}{\partial x}$ is also used. The transpose of a matrix $\B{A}$ is denoted as $\B{A}^T$ and $\B{F}\cdot\B{G}\equiv\B{F}^T\B{G}$ represents the scalar product of two vectors in the corresponding Euclidean space. 
We introduce the time-spatial inner product of two vector-valued functions $\B{h},\B{g}:\Omega\times [0,T]\to \R^q$ 
%defined on a  space domain $\Omega$ and a time interval $[0,T]$
%
\begin{equation}
\label{eq:inner}
\langle \B{h},\B{g}\rangle_{\Omega \times T}:=\int_0^T\int_{\Omega}\B{h}^T(\mathbf{x},t) \B{g}(\mathbf{x},t)d\mathbf{x}dt\,.
%=: \int_0^T \langle h, g \rangle_\Omega \, dt\,.
\end{equation}
\subsection{The misfit function for the FWI}
Given the observed data $\mathbf{d}=\mathbf{d}(\B{y}_r,t)$  at receivers placed in the
positions $\mathbf{y}_r$, $r=1,\cdots,N$, the misfit functional is defined by
\begin{eqnarray}
\label{eq:misfit}
\chi(\bW,\bM):=\int_0^T\int_\Omega \frac{1}{2} \| \mathcal{R}\bW(\bM; \B{x},t) -\B{d}(\B{x},t)\|^2 \mu(dx)dt=:
\langle \Phi, \B{1}\rangle_{\Omega\times T}
%f(\bM,\bW,t) dt\,,\\
%f(\bW(\bM),t):= \int_\Omega \frac{1}{2} \| \mathcal{R}\bW(\bM, \B{x},t) -\B{d}(\B{x},t)\|^2 \mu(dx)\,,
\label{eq:def_f}
\end{eqnarray}
where $\mu(dx)$ is a measure representing the receivers, typically the sum of Dirac masses 
$\mu(dx)=\sum_{r=1}^N \delta_{y_r}$ concentrated at points $y_r\in\Omega$. For the purpose of this derivation
we can assume that these are appropriately regularized so the integral in \eqref{eq:def_f} is well defined.
We denote $\mathcal{R}$ the restriction operator that sets the 'unmeasurable' or unavailable components 
in  $\bW$  to be zero. 
The given vector $\B{d}$ has the same dimension and format as $\bW$ by augmenting the data vector with 
zeros for the components which have no available data. Therefore, the components in the 
discrepancy $\mathcal{R}\bW-\B{d}$ which correspond to the non-measurable entries in $\bW$ must equal to zero.  

We consider the following constrained optimization problem for the property field $\B{M}$ in a suitable space
of admissible functions
\begin{eqnarray}
%\begin{aligned}
&&\min_\bM ~ \chi(\bW,\bM)\,, \label{eq:optdef}\\
&\mbox{subject to:}\; &~\pr_t\bW + \mathcal{A}(\bM,\bW,\nabla_{\bx}\bW) = 0\,,\;\;\;\;\mbox{ (PDE),} \label{eq:constraintPDE}\\
%\mathcal{F}(\mathbf{M},\mathbf{W}, {\mathbf{W}_t}, {\bW_x}, {\bW_z},\B{S})=0\,,\;\;\;\;\mbox{ (PDE),}\\
&& \B{g}(\bM,\bx):=\bW(\bx,0;\bM)-\boldsymbol{\phi}(\bx;\bM)=\B{0}\,,\;\;\;\; \mbox{ (IC)} \label{eq:constraintIC}
%\end{aligned}
\end{eqnarray}
where the mapping 
$\mathcal{A}:(\mathbf{m},\mathbf{w},\boldsymbol{\zeta})\in\R^m\times\R^n\times\R^{3\times n}\mapsto\R^n$
is defined as in \eqref{eq:direqn}, specifically
\begin{equation}\label{eq:Adef}
\mathcal{A}(\bM,\bW,\nabla_{\bx}\bW) =\B{A}(\bM(\bx))\bW_x+\B{B}(\bM,\bx)\bW_z+\B{H}_1(\bM,\bx)\bW+\B{H}_2(\bM,\bx)\B{S}(\bM,\bx,t)\,.
\end{equation}
%is defined  \eqref{eq:direqn} and ${\bW_x},{\bW_z}$ 
%are the partial derivatives of $\bW$ with respect to $x$ and $z$, respectively. 
 For the sake of clarity we also abbreviate the equation \eqref{eq:constraintPDE} as
 $\mathcal{F}(\bM,\partial_t\bW,\bW,\nabla_{\bx}\bW) = 0$.
 %clarity of presentation, we  rewrite  from  \eqref{eq:direqn}  the expression of $\mathcal{F}$ 
% by separating the spatial partial derivatives into two terms and use new symbols for the corresponding coefficients
% \beq
% \mathcal{F}=\pr_t \bW +\B{A}(\bM(\bx))\bW_x+\B{B}(\bM,\bx)\bW_z+\B{H}_1(\bM,\bx)\bW+\B{H}_2(\bM,\bx)\B{S}(\bM,\bx,t)\,,
% \label{F_form}
% \eeq
% or expressing the constraint $\mathcal{F}=0$ in a more compact form 
% \begin{equation}\label{F_formComp}
% \pr_t \bW + \mathcal{A}(\bM,\bW,\bW_x,\bW_z) = 0\,.
% \end{equation}

\subsection{Adjoint problem and variation of the misfit function}
A gradient type minimization iteratively updates values of the field $\B{M}(x)$ in the feasible 
set using descent directions derived from the variation  $\delta_\sbM\chi$ of the misfit function. 
The first task is thus derivation of admissible  directions that belong to the tangent space of the constraints
\eqref{eq:constraintPDE}-\eqref{eq:constraintIC}.
Given a perturbation $\delta\mathbf{M}$, the solution of the forward problem results in  a perturbation
(variation) $\delta_\sbM\bW$ and thus a perturbation $\delta_\sbM\chi$ of the misfit function along
the constraint \eqref{eq:constraintPDE}. 

Assuming sufficient regularity of the misfit functional $\chi$ and using the Taylor expansion for $\chi$, 
the infinitesimal variation in the direction of $\delta\bM$ along the constraints is given by
the Fr\'echet derivative (taking into account that $\Phi$ does not depend explicitly on $\bM$)
\begin{equation}\label{eq:Tayl misfitbis}
\delta_\sbM \chi:=d_{\sbM}\chi[\delta\B{M}] = 
\langle \prg_{\mathbf{w}}\Phi,\delta_\sbM\B{W}\rangle_{\Omega\times T}\,,
\end{equation}
where the function $\B{V}:=\delta_\sbM\B{W}$ solves the variation equation
for the constraint \eqref{eq:constraintPDE} under the perturbation $\delta\bM$, i.e.,
\begin{equation}\label{eq:varW}
\begin{split}
  \pr_t \B{V} + \nabla_{\mathbf{w}} \mathcal{A}(\bM,\bW,\nabla_\bx\bW)\cdot\B{V} 
     + \nabla_{\boldsymbol{\zeta}} \mathcal{A}(\bM,\bW,\nabla_\bx\bW) : \nabla_\bx\B{V}
   &= -\nabla_{\mathbf{m}}\mathcal{A}(\bM,\bW,\nabla_\bx\bW)\cdot \delta\bM\,, \\
   \B{V}(x,0) &= \nabla_{\mathbf{m}}\B{g}(\bM,\bx)\,,
\end{split}
\end{equation}
along with the solution $\bW(x,t;\bM)$ of \eqref{eq:constraintPDE}-\eqref{eq:constraintIC}.
% \begin{eqnarray*}
% &=& \sum_{r=1}^N \int_0^T\left(\mathcal{R}\bW(\bM,\B{y}_r,t)-\bd(\B{y}_r,t)  \right)\cdot \mathcal{R}([d_\sbM \bW ]\delta\bM(\B{y}_r)) dt\nonumber\\
% &=&\sum_{r=1}^N \int_0^T\left(\mathcal{R}\bW(\bM,\B{y}_r,t)-\bd(\B{y}_r,t)  \right)\cdot \mathcal{R}(\delta_\sbM \B{W}(\B{y}_r)) dt.
% %+\left \langle\frac{\partial \chi(\mathbf{m}_0,\mathbf{W}_0)}{\partial \mathbf{W}},\delta \mathbf{W} \right \rangle_T
% %+[\partial_\sbM \chi(\mathbf{M},\mathbf{W})] \delta \mathbf{M}.
% \end{eqnarray*}
% \begin{eqnarray*}
% \label{eq:Tayl misfit}
% \delta_\sbM \chi&:=&\chi\left(\bW(\bM+\delta \bM)\right)-\chi(\bW(\bM))\nonumber\\
% &=& \sum_{r=1}^N \int_0^T\left(\mathcal{R}\bW(\bM,\B{y}_r,t)-\bd(\B{y}_r,t)  \right)\cdot \mathcal{R}([d_\sbM \bW ]\delta\bM(\B{y}_r)) dt\nonumber\\
% &=&\sum_{r=1}^N \int_0^T\left(\mathcal{R}\bW(\bM,\B{y}_r,t)-\bd(\B{y}_r,t)  \right)\cdot \mathcal{R}(\delta_\sbM \B{W}(\B{y}_r)) dt.
% %+\left \langle\frac{\partial \chi(\mathbf{m}_0,\mathbf{W}_0)}{\partial \mathbf{W}},\delta \mathbf{W} \right \rangle_T
% %+[\partial_\sbM \chi(\mathbf{M},\mathbf{W})] \delta \mathbf{M}.
% \end{eqnarray*}
Thus solving \eqref{eq:constraintPDE}-\eqref{eq:constraintIC}, \eqref{eq:varW} for the functions $(\bW,\B{V})$ and taking into account the specific form of $\Phi$ in \eqref{eq:misfit} we have
%Because of the definition of $\mathcal{R}$ and the format of data $\B{d}$, we have 
\begin{eqnarray}\label{eq:Tayl misfit}
\delta_\sbM \chi= \sum_{r=1}^N \int_\Omega \int_0^T\left(\mathcal{R}\bW(\bM,\B{x},t)-\bd(\B{x},t)  \right)^T \B{V} \,\delta_{y_r}(\B{x}) d\bx\,dt\,.
\end{eqnarray} 
%where $\delta{(\B{y}_r)}$ is the Dirac measure concentrated at $\B{y}_r$, $r=1,\cdots,N$.
\iffalse%=================================>
Note that 
\[
\frac{\pr f}{\pr W_j}=\sum_{r=1}^N \mathcal{R}W_j(\bM,\B{y}_r,t) -d_j(\B{y}_r,t),\quad \prg_\sbW f = \sum_{r=1}^N \int_\Omega \left\{ (\mathcal{R}\bW(\bM,\B{x},t) -\B{d}(\B{r}_r,t) )\delta(\B{y}_r) \right\}^T\, d\bx.
\]

Since $F(\mathbf{m},\mathbf{W})=0$, we can know that
\begin{equation}
\label{eq:DeltaF Tayl}
\mathbf{0}=F(\mathbf{m}_0+\delta\mathbf{m},\mathbf{W}_0+\delta\mathbf{W})=F(\mathbf{m}_0,\mathbf{W}_0)+\frac{\partial F(\mathbf{m}_0,\mathbf{W}_0) }{\partial \mathbf{W}}\delta \mathbf{W}+\frac{\partial F(\mathbf{m}_0,\mathbf{W}_0)}{\partial \mathbf{m}}\delta \mathbf{m}.
\end{equation}

Therefore, $\delta \B{W}$ and $\delta \B{M}$ are related by
\begin{equation}
\label{eq:DF Wm}
\frac{\partial F(\mathbf{m},\mathbf{W}) }{\partial \mathbf{W}}\delta \mathbf{W}=-\frac{\partial F(\mathbf{m},\mathbf{W})}{\partial \mathbf{m}}\delta \mathbf{m}.
\end{equation}
\fi%<=============================================================
An alternative way to compute the variation $\delta_\sbM\chi$ is to use the formulation of
the constrained minimization problem \eqref{eq:optdef} with the Lagrangian functional
\begin{equation}\label{eq:Lagrangian}
\mathcal{L}(\mathbf{M},\B{W},\blam,\bmu)=
\langle \Phi,\B{1}\rangle_{\Omega\times T}
+\langle \blam,\mathcal{F}\rangle_{\Omega \times T}+ \int_\Omega \bmu^T \B{g} d\bx\,,
\end{equation}
where we introduced
the Lagrange multipliers $\blam(\bx,t)\in\R^n$ and $\bmu(\bx)\in\R^n$. In this formulation we can vary $\delta\bM$ and $\delta\bW$ independently thus obtaining the Fr\'echet derivative $d\mathcal{L}$ at a point $(\bM,\bW,\blam,\bmu)$ acting on the increment $(\delta\bM,\delta\bW)$, i.e.,
\begin{equation}\label{eq:dLagrangian}
    \begin{split}
        d\mathcal{L}[(\delta\bM,\delta\bW)] =& \la \nabla_{\mathbf{w}}\Phi,\delta\bW\ra_{\Omega \times T} +
        \la \blam,d_\sbW\mathcal{F}[\delta\bW]\ra_{\Omega \times T} + \la \blam,d_\sbM\mathcal{F}[\delta\bM]\ra_{\Omega \times T} \\
        &+ \int_\Omega \bmu^T d_\sbW\B{g}[\delta\bW] \,d\bx + \int_\Omega \bmu^T d_\sbM\B{g}[\delta\bM] \,d\bx\,.
    \end{split}
\end{equation}
From the definition of the mapping $\mathcal{F} = \partial_t \bW + \mathcal{A}$ and $\B{g}$ we have
\begin{equation*}
    \begin{split}
        d_\sbW\mathcal{F}[\delta\bW] &= \partial_t\delta\bW + \nabla_{\mathbf{w}}\mathcal{A}\cdot\delta\bW + 
                                        \nabla_{\boldsymbol{\zeta}}\mathcal{A}:\nabla_\bx\delta\bW \\
        d_\sbM\mathcal{F}[\delta\bM] &= \nabla_{\mathbf{m}}\mathcal{A}\cdot\delta\bM\,,\\
        d_\sbW\B{g}[\delta\bW] &= \delta\bW\,,\;\;\;\;\;\;d_\sbM\B{g}[\delta\bM] = -\nabla_{\mathbf{m}}\boldsymbol{\phi}\cdot\delta\bM\,.
    \end{split}
\end{equation*}
Combining the first and second term in \eqref{eq:dLagrangian} and integrating by parts in the expression for $d_\sbW\mathcal{F}[\delta\bW]$ we obtain
\[
\begin{split}
\la \nabla_{\mathbf{w}}\Phi,\delta\bW\ra_{\Omega \times T} +
        \la \blam,d_\sbW\mathcal{F}[\delta\bW]\ra_{\Omega \times T} &= 
\la -\pr_t\blam +\blam^T\nabla_{\mathbf{w}}\mathcal{A} - \DIV(\blam^T\nabla_{\boldsymbol{\zeta}}\mathcal{A}) + \nabla_{\mathbf{w}}\Phi,\delta\bW\ra_{\Omega\times T} \\
&+ \int_\Omega (\blam^T(\bx,T)\delta\bW(\bx,T) - \blam^T(\bx,0)\delta\bW(\bx,0)\,d\bx\,,
\end{split}
\]
where we set $\blam(\bx,t)=0$ for $\bx\in\partial\Omega$ to remove the boundary terms from the integration by parts.
We denote the solution $\blam_*(\bx,t)$ of the adjoint problem
\begin{equation}\label{eq:adjProb}
\begin{split}
   &\pr_t\blam -\blam^T\nabla_{\mathbf{w}}\mathcal{A} + \DIV(\blam^T\nabla_{\boldsymbol{\zeta}}\mathcal{A}) = \nabla_{\mathbf{w}}\Phi\,,\;\;\;\bx\in\Omega\,,\;\;t<T \\
   &\blam(\bx,T) = 0\,,\;\;\;\;\blam(\bx,t)|_{\partial\Omega} = 0\,,
\end{split}
\end{equation}
and set the multiplier $\bmu_*(\bx) = \blam_*(x,0)$. Hence we have
\[
\int_\Omega \bmu^T d_\sbW\B{g}[\delta\bW] \,d\bx + \int_\Omega \bmu^T d_\sbM\B{g}[\delta\bM] \,d\bx = \int_\Omega \blam_*(x,0)\cdot\delta\bW(\bx,0)\,d\bx - \int_\Omega \blam_*(x,0)\cdot(\nabla_{\mathbf{m}}\boldsymbol{\phi}\delta\bM(\bx))\,d\bx\,.
\]
Substituting back to \eqref{eq:dLagrangian} we obtain the expression for the variation $\delta_\sbM\chi$ of the misfit functional $\chi$ evaluated at the solution $\bW_*$ of the forward problem \eqref{eq:constraintPDE}-\eqref{eq:constraintIC} as
\begin{equation}\label{eq:varCHI}
    \delta_\sbM\chi = d\mathcal{L}(\bM,\bW_*,\blam_*,\bmu_*)[\delta\bM] = \la\blam_*,\nabla_{\mathbf{m}}\mathcal{A}\delta\bM\ra_{\Omega\times T} - \int_\Omega \blam_*(x,0)\cdot\nabla_{\mathbf{m}}\boldsymbol{\phi}\delta\bM(\bx)\,d\bx\,.
\end{equation}
Returning back to the specific form of $\mathcal{A}$, cf.\eqref{eq:Adef}, and $\Phi$, cf. \eqref{eq:misfit}, we have $\blam_*$ solving
\begin{empheq}[left={\mbox{(adjoint problem)}\empheqlbrace}]{align*}
&{\pr_t \blam}+ \pr_x \left( \B{A}^T \blam\right)+\pr_z \left( \B{B}^T \blam\right)-\B{H}_1^T\blam=
\sum_{r=1}^N  \left[\B{R}\bW(\B{x},t)- \B{d}(\B{x},t)\right]\delta(\B{y}_r)
\label{adj_eq},\\
&\blam(\bx,T)=\B{0}, \quad \blam(\bx,t)|_{\pr\Omega}=\B{0}\,, 
%\quad \bmu(\bx)=\blam(\bx,0).
\end{empheq}
and $\bmu_*(\bx)=\blam_*(\bx,0)$.
%
%%%
Solving the adjoint problem for $\blam_*$ eliminates the need for computing the solution $\B{V}={\delta_\sbM \bW}$ of the variation equation \eqref{eq:varW} and it gives an explicit expression for the descent direction in the gradient-type minimization algorithm. From \eqref{eq:varCHI} we have
\beqas
 \delta_\sbM \chi&=&\int_\Omega  \sum_{k=1}^q  \delta M_k(\bx)  \int_0^T  \blam_*^T \left\{ \pr_{m_k} \B{A}\bW_x  
 + \pr_{m_k} \B{B}\bW_z+\pr_{m_k} \B{H}_1\bW+\pr_{m_k} \B{H}_2\B{S}\right.\\
 && \left. +\B{H}_2\pr_{m_k} \B{S} -\pr_{m_k} \boldsymbol{\phi}(\bx) \right\} dt\, d\bx\,,
 \eeqas
where $\delta M_k$, $k=1,\dots,q$ denotes components of the increment $\delta\bM$. Thus we have an explicit formula for the directional derivative $\delta_\sbM\chi = \la\B{G},\delta\bM\ra_\Omega$ where 
the vector- valued function $\B{G}(\bx)\in\R^q$ represents the Fr\'echet derivative and is given component-wise as
\begin{equation}\label{eq:gradient}
[\B{G}]_k= \int_0^T  \blam_*^T \left\{ \pr_{m_k} \B{A}\bW_x  
 + \pr_{m_k} \B{B}\bW_z+\pr_{m_k} \B{H}_1\bW+\pr_{m_k} \B{H}_2\B{S}+\B{H}_2\pr_{m_k} \B{S} -\pr_{m_k} \boldsymbol{\phi}(\bx) \right\} dt\,.
 \end{equation}
Therefore, the direction of descent for minimizing the misfit function $\chi$ and updating the field $\bM$ is given by $-\B{G}$.
 \section{\label{Algorithm}Numerical method for computing $\{\vartheta_k\}$ and $\{r_k\}$}
To make this paper selt-contained, we summarize in this section the numerical method for poles $\vartheta_k$ and residues $r_k$, $k=1,\dots,M$, by using the two-sided residue method; the calculation should be carried out in an arbitrary precision arithmetic system such as Advanpix (a multiprecision Matlab toolbox) to ensure the returned poles and residues are of correct signs \cite{ou2018augmented}. The two-sided residue interpolation method involves solving a linear system of size $M$ from given values of $\alpha$ at $M$ distinct frequencies and  automatically interpolates $\alpha$ at $\omega=0$ and $\omega=\infty$.  Depending on the spectral content of the source term $\B{S}$, the interpolation grid is either equally spaced or log-spaced in the frequency range.

Given the M interpolation data $(z_k,u_k,v_k)\in \mathbb{C}^+\times\mathbb{C}^{p\times q}\times\mathbb{C}^{p\times q}$, where $\mathbb{C}^+$ denotes the upper-half complex plane, we seek a $p\times p$ matrix valued function $G(z)$ of the form
\begin{equation}
\label{G def}
G(z)=\int_0^{\infty}\frac{d\mu(t)}{t-z},
\end{equation}
where $\mu$ is a positive $p\times p$ matrix valued measure, such that
\begin{equation}
\label{G cond}
G(z_i)u_i=v_i, i=1,\cdots,M.
\end{equation}
Define the Hermitian matrices $S_1$ and $S_2$ via
\begin{equation}
\left(S_{1}\right)_{i j}:=\frac{u_{i}^{*} v_{j}-v_{i}^{*} u_{j}}{z_{j}-\overline{z_{i}}},
~\left(S_{2}\right)_{i j} :=\frac{z_{j} u_{i}^{*} v_{j}-\overline{z_{i}} v_{i}^{*} u_{j}}{z_{j}-\overline{z_{i}}},~ i, j=1, \ldots, M,
\end{equation}
then $S_1$ and $S_2$ are positive semidefinite. Here the sup-script  $*$ denotes the conjugate transpose operator.  Conversely, if $S_1$ is positive and $S_2$ is positive semidefinite, then
\begin{equation}
\begin{aligned}
\label{eq:GDefAppro}
G(z) :=&-C_{+}\left(z S_{1}-S_{1} A-C_{+}^{*} C_{-}\right)^{-1} C_{+}^{*}=C_{+}\left(S_{2}-z S_{1}\right)^{-1} C_{+}^{*}\\
&=\sum_{k=1}^{qM}\left(\frac{1}{d_{k}-z}\right) C_{+} S_{1}^{-\frac{1}{2}} \boldsymbol{x}_{k} \boldsymbol{x}_{k}^{*} S_{1}^{-\frac{1}{2}} C_{+}^{*},
\end{aligned}
\end{equation}
is a solution to the interpolation problem, where $\boldsymbol{x}_j$ is the eigenvector of $S_1^{\frac{-1}{2}}S_2S_1^{\frac{-1}{2}}$ corresponding to the eigenvalue $d_k$ and $C_{-} :=\left({u_{1}}, {\cdots},{u_{M}}\right)$, $C_{+} := \left({v_{1}}, {\cdots}, {v_{M}}\right)$, $A :=\operatorname{diag}\left(z_{i}\right)$.
Specializing the results to the function $D(s)$, with change of variables $-i \omega=:s=-\frac{1}{z}$, we have $p=q=1$  and
\begin{equation}
\label{eq:D Twosided}
D(s)=\int_{0}^{\theta} \frac{d \sigma(t)}{1+s t}=(-z) \int_{0}^{\theta} \frac{d \sigma}{t-z},~s=-\frac{1}{z},
\end{equation}
then
\begin{equation}
\begin{aligned}
D(s)-\alpha_{\infty}&=(-z)\left(\int_{0}^{\theta} \frac{d \sigma(t)}{t-z}-\frac{\alpha_{\infty}}{-z}\right)\\
&=(-z)\left(\int_{0}^{\theta} \frac{d \sigma(t)}{t-z}-\int_{0}^{\theta} \frac{\alpha_{\infty} \sigma(t)}{t-z}\right)
:=(-z)G_{new}(z),
\end{aligned}
\end{equation}
which is Stieltjes function owing to $\sigma$ has a Dirac measure of strength $\alpha_{\infty}$.
Using (\ref{eq:GDefAppro}), we can compute the poles and residues for $G_{new}$ as follows, which amounts to solving a generalized eigen value problem.
\begin{eqnarray}
-i\omega_k=:s_{k} &=& -\frac{1}{z_{k}},~ u_{k}=\frac{1}{s_{k}},~ k=1 \ldots M, \\
v_{k} &=& D\left(s_{k}\right)-\alpha_{\infty}, k=1\ldots M,\\
\left(S_{1}\right)_{pq} &=& \frac{-s_{q} D\left(s_{q}\right)+s_{p}^{*} D^{*}\left(s_{p}\right)}{s_{p}^{*}-s_{q}}-\alpha_{\infty},~ p,\,q=1 \ldots M,\\
\left(S_{2}\right)_{pq} &=& \frac{-D\left(s_{q}\right)+D^{*}\left(s_{p}\right)}{s_{q}-s_{p}^{*}},~ p,\,q=1 \ldots M,\\
(-z)G_{new}(z) &=& {C}_+ \left(S_1+s S_2\right)^{-1} {C}_+^*.
\end{eqnarray}

In terms of the generalized eigenvalues/eigenvectors $[\B{V},\Phi]:=eig(S_1,S_2)$, where $\B{V}$ is the matrix of generalized  vectors and $\Phi$ the diagonal matrix of generalized eigenvalues such that
\beq
S_1 \B{V}=S_2 \B{V} \Phi,
\label{Algo_2}
\eeq
and taking into account of the simultaneous diagonalization property
\beqa
\begin{aligned}
\B{V}^* S_1 \B{V}&=\Phi,\quad \B{V}^* S_2 \B{V}&=\B{I},
\end{aligned}
\eeqa
we obtain
\beq
D(s)=\alpha_\infty+\sum_{k=1}^{N} \frac{C_+ \B{V}_1(:,k)\B{V}_1(:,k)^* C_+^*}{s+ \Phi(k,k)}.
\label{two-sided}
\eeq
The poles $\vartheta_k$ and residues $r_k$, $k=1,\dots,M$, are given by
\begin{eqnarray}
\vartheta_k &=& -  \Phi(k,k), \\
r_k &=&C_+ \B{V}(:,k)\B{V}(:,k)^* C_+^*.
\end{eqnarray}
\section{Conclusion and future work}\label{conclu}
In this paper, we consider the time-domain FWI for the poroelastic wave equations, whose dispersive and dissipative  behaviors are encoded in the auxiliary variables $\Theta_k$, $k=1,\dots,N$. Since all possible memory kernel functions that satisfy the causality conditions can be represented as an integral presented here, the FWI framework proposed here is general enough to handle all porous material with pore geometry regular enough for the spectral theory of Stokes equation to hold true. The energy analyses for the popular Biot-JKD equations and for the most general augmented Biot equations have been also presented in this paper; they play an important  role in understand the stability of the adjoint problem, which is derived in this paper. The adjoint problem involves solving a PDE backward in time for the Lagrange multiplier function $\blam(\bx,t)$. In this framework, the minimization of the misfit function can be solved by using the gradient calculated from using the solutions of the forward and the adjoint problem. 

This paper lays the foundation for a thorough mathematical and numerical analysis for the time-domain FWI of dispersive and dissipative poroelastic wave equations.  On the mathematical side, the future work includes the analysis of the minimization problem such as the uniqueness and existence of the minimizer convergence rate, the choice of regularization terms and the error estimates. The question of uniqueness of the minimizer  is highly non-trivial because a given memory kernel can be approximated equally well by different choices of pole-residue sets; this line of research will involve the definition of equivalent classes of pole-residue and factor it into the iteration steps. We would also like to expand our energy analysis to the case where the initial conditions of  the fluid relative velocity $\B{q}$ is not zero. On the numerical side, the immediate future work includes the development of an efficient solver for the adjoint problem and the strategy for handling the time integral for computing the gradient $\B{G}$; for handing the memory demand, a strategy such as the CARFS \cite{doi:10.1190/geo2016-0082.1} will be considered.
\vspace{0.2in}

%%%
\noindent{{\bf Acknowledgement:} The work of MYO is partially sponsored by US NSF grants DMS-1413039 and DMS-1821857. The work of J. Xie is partially supported by China NSFC Grants No. 12171274) and No. 11871139.}
%%%%%%%
\bibliographystyle{plain}
%\bibliography{/Users/mou/Documents/library.bib}
%\bibliography{library.bib}

%%%%%
\end{document}